\documentclass[a4paper,12pt,reqno]{amsart}
\usepackage{amsmath}
\usepackage{amsfonts}
\usepackage{amssymb}
\usepackage{amsthm}
\usepackage{amscd}
\usepackage{color}
\usepackage{verbatim}
\usepackage{eucal,url,amssymb,stmaryrd,enumerate,amscd,}
\usepackage{amsfonts}

\usepackage{amsmath,amsthm,amssymb,amscd,enumerate,eucal,url,stmaryrd}
\usepackage{verbatim}

\newcommand{\RR}{\mathbb{R}}

\numberwithin{equation}{section}

\theoremstyle{plain}
\newtheorem{proposition}{Proposition}[section]
\newtheorem{theorem}[proposition]{Theorem}
\newtheorem{lemma}[proposition]{Lemma}

\theoremstyle{definition}
\newtheorem{definition}[proposition]{Definition}

\theoremstyle{remark}
\newtheorem{remark}[proposition]{Remark}

\begin{document}

\title[Formality of $7$-dimensional $3$-Sasakian manifolds]{Formality of $7$-dimensional $3$-Sasakian manifolds}

\author[M. Fern\'andez]{Marisa Fern\'{a}ndez}
\address{Universidad del Pa\'{\i}s Vasco,
Facultad de Ciencia y Tecnolog\'{\i}a, Departamento de Mate\-m\'a\-ticas,
Apartado 644, 48080 Bilbao, Spain}
\email{marisa.fernandez@ehu.es}

\author[S. Ivanov]{Stefan Ivanov}
\address{University of Sofia "St. Kl. Ohridski",
Faculty of Mathematics and Informatics\\
Blvd. James Bourchier 5\\
1164 Sofia, Bulgaria}
\address{Institute of Mathematics and Informatics, Bulgarian Academy of
Sciences}
\email{ivanovsp@fmi.uni-sofia.bg}

\author[V. Mu\~{n}oz]{Vicente Mu\~{n}oz}
\address{Facultad de Ciencias Matem\'aticas, Universidad
Complutense de Madrid, Plaza de Ciencias
3, 28040 Madrid, Spain}
\address{Instituto de Ciencias Matem\'aticas (CSIC-UAM-UC3M-UCM),
C/ Nicolas Cabrera 15, 28049 Madrid, Spain}
\email{vicente.munoz@mat.ucm.es}



\begin{abstract}
We prove that any simply connected compact  $3$-Sasakian manifold, of dimension seven,
is formal if and only if its second Betti number is $b_2<2$. In the opposite, we show
an example of a $7$-dimensional Sasaki-Einstein manifold,
with second Betti number $b_{2}\geq 2$, which is formal.  Therefore, such an example does not admit
any  $3$-Sasakian structure. Examples of $7$-dimensional simply connected compact formal Sasakian manifolds,
with $b_{2}\geq 2$, are also given.
\end{abstract}

\maketitle

%%%%%%%%%%%%%%%%%%%%%%%%%%%%%%%%%%%%%%%%%%%%%%%%%%%%%%%
%%%%%%%%%%%%%%%%%%%%%%%%%%%%%%%%%%%%%%%%%%%%%%%%%%%%%%%
%%%%%%%%%%%%%%%%%%%%%%%%%%%%%%%%%%%%%%%%%%%%%%%%%%%%%%%
\section{Introduction}\label{sec:intro}
%%%%%%%%%%%%%%%%%%%%%%%%%%%%%%%%%%%%%%%%%%%%%%%%%%%%%%%
%%%%%%%%%%%%%%%%%%%%%%%%%%%%%%%%%%%%%%%%%%%%%%%%%%%%%%%

A Riemannian manifold $(N, g)$, of dimension $2n+1$, is Sasakian if its metric cone
$(N\times{\mathbb{R}}^+, 
g^c =  t^2 g+dt^2)$ is K\"ahler. If in addition the metric $g$ is Einstein, then $(N, g)$
is said to be a {\em Sasaki-Einstein} manifold. In this case, the cone metric $g^c$ is Ricci flat.
Sasakian geometry is the odd-dimensional counterpart to K\"ahler geometry. Indeed, just as K\"ahler geometry
 is the intersection of complex, symplectic and Riemannian geometry, Sasakian geometry is
 the intersection of normal, contact and Riemannian geometry.

Sasakian structures can be also defined in terms of strongly pseudo
convex CR-structures, namely a strongly pseudo convex CR-structure
is Sasakian exactly when the Webster torsion vanishes (see e.g.\
\cite{DT}).

One of the results of Deligne, Griffiths, Morgan and Sullivan states that any compact K\"ahler manifold is formal  \cite{DGMS}.
However, the first and third authors in \cite{BFMT} have proved that the formality is not an obstruction
to the existence of Sasakian structures even on  simply connected manifolds. Indeed,
 examples of $7$-dimensional simply connected compact non-formal
Sasakian manifolds, with second Betti number $b_2(N)\,\geq 2$, are constructed in \cite{BFMT}.
(Note that any $7$-dimensional simply connected compact manifold  with $b_2\,\leq 1$ is formal,
see section \ref{min-models} for details.) Nevertheless, in  \cite{BFMT} it is also proved that all higher Massey products are trivial
on any compact Sasakian manifold.

We remind that a $3$-Sasakian manifold is a Riemannian manifold
$(N, g)$, of dimension $4n+3$, such that  its cone
$(N\times{\mathbb{R}}^+, g^c =  t^2 g\,+\,dt^2)$ is hyperk\"ahler,
and so the holonomy group of $g^c$ is a subgroup of
$\mathrm{Sp}(n+1)$. Thus $3$-Sasakian manifolds are automatically
Sasaki-Einstein with positive scalar curvature
\cite{Kas}. Consequently, a complete $3$-Sasakian manifold is
compact with finite fundamental group due to the Myers' theorem.
The hyperk\"ahler structure on the cone induces a $3$-Sasakian
structure on the base of the cone. In particular, the triple of
complex structures gives rise to a triple of Reeb vector fields
$(\xi_1,\xi_2, \xi_3)$
 whose Lie brackets give a copy of the Lie algebra $\mathfrak{su}(2)$. 
 
A $3$-Sasakian manifold $(N, g)$ is said to be {\em regular} if the vector fields $\xi_i$ $(i=1, 2, 3)$ are complete and 
the corresponding $3$-dimensional foliation is regular, so that the space of leaves is a smooth 
$4n$-dimensional manifold $M$. 
Ishihara and Konishi in \cite{Ishihara-Konishi} noticed that the induced metric on the 
latter is quaternionic K\"ahler  with positive scalar curvature. 
Conversely, starting with a quaternionic K\"ahler manifold $M$ of positive scalar curvature, 
the manifold $M$ can be recovered as the total space of a bundles naturally associated to $M$.

The above situation has been generalized to the orbifold category by Boyer, Galicki and Mann 
in \cite{BGMa} (see also \cite{BG-2}). In fact, if the $3$-Sasakian manifold is compact, then the
 Reeb vector fields $\xi_i$ are complete, the corresponding $3$-dimensional foliation has compact leaves and the space of leaves is
 a compact quaternionic
K\"ahler  manifold or orbifold. We recall that a
$4n$-dimensional $(n>1)$ Riemannian manifold/orbifold is
quaternionic K\"ahler if it has holonomy group contained in
$\mathrm{Sp}(n)\mathrm{Sp}(1)$, and a 4-dimensional quaternionic K\"ahler
manifold/orbifold is a self-dual Einstein  Riemannian
manifold/orbifold.

The $3$-Sasakian structures can be considered also from a
sub-Riemannian point of view \cite{Rizzi} by using quaternionic
contact structures \cite{Biq}. A $3$-Sasakian manifold is precisely
a quaternionic contact manifold with vanishing Biquard torsion and
positive (quaternionic) scalar curvature \cite{IMV}.

Important results on the topology of a compact $3$-Sasakian manifold
were proved by Galicki and Salamon in \cite{GS}. There it is proved that the odd Betti numbers
$b_{2i+1}$ of such a manifold, of dimension $4n+3$, are all zero for $0\leq i \leq n$.
Moreover, for regular compact $3$-Sasakian manifolds many topological properties are known
(see  \cite[Proposition 13.5.6 and Theorem 13.5.7]{BG}). For example, such a manifold is simply connected unless $N\,=\,{\mathbb{RP}}^{4n+3}$.
Also, using the results of Lebrun and Salamon \cite{LeBrun-Salamon} about the topology of positive quaternionic K\"ahler manifolds,
Boyer and Galicki in \cite{BG} show interesting relations among the Betti numbers of regular compact $3$-Sasakian manifolds;
in particular $b_{2}\,\leq\,1$.
Nevertheless, in \cite{BGMR} it is proved that there exist many $3$-Sasakian manifolds, of dimension 7,
with arbitrary second Betti number.
The goal of this note is to prove the following.

\begin{theorem}\label{th:main}
Let $(N, g)$ be a simply connected compact $3$-Sasakian manifold, of dimension $7$. Then, $N$ is formal
if and only if its second Betti number $b_2(N)\leq 1$.
\end{theorem}

On the other hand, we prove that formality allows one to
distinguish $7$-dimensional Sasaki-Einstein manifolds  which admit $3$-Sasakian
 structures from those which do not. In fact, we show an example
 of a $7$-dimensional regular simply connected Sasaki-Einstein manifold,
with second Betti number $b_{2}\geq 2$, which is formal.
Thus, Theorem \ref{th:main} implies that such a manifold does not admit any
$3$-Sasakian structure. Our example is the total
space of an $S^1$-bundle over a positive K\"ahler Einstein
6-manifold which is the blow up of the complex projective space
${\mathbb{CP}}{}^3$ at four points.

%%%%%%%%%%%%%%%%%%%%%%%%%%%%%%%%%%%%%%%%%%%%%%%%%%%%%%%
%%%%%%%%%%%%%%%%%%%%%%%%%%%%%%%%%%%%%%%%%%%%%%%%%%%%%%%
%%%%%%%%%%%%%%%%%%%%%%%%%%%%%%%%%%%%%%%%%%%%%%%%%%%%%%%
\section{Minimal models and formal manifolds}\label{min-models}
%%%%%%%%%%%%%%%%%%%%%%%%%%%%%%%%%%%%%%%%%%%%%%%%%%%%%%%
%%%%%%%%%%%%%%%%%%%%%%%%%%%%%%%%%%%%%%%%%%%%%%%%%%%%%%%

In this section, we recall concepts about minimal models and formality from
\cite{DGMS, FHT, FM}.

Let $(\mathcal{A}, d_{\mathcal{A}})$ be a {\it differential graded commutative algebra} over the real numbers $\mathbb{R}$ (in the sequel,
we shall say just a differential algebra), that is, $\mathcal{A}$ is a
graded commutative algebra over $\mathbb{R}$ equipped with
 a differential $d_{\mathcal{A}}$ which is a derivation, i.e.
$d_{\mathcal{A}}(a\cdot b) = (d_{\mathcal{A}}a)\cdot b +(-1)^{|a|} a\cdot (d_{\mathcal{A}}b)$, where
$|a|$ is the degree of $a$. Given a differential algebra $(\mathcal{A},\,d_{\mathcal{A}})$, we
denote its cohomology by $H^*({\mathcal{A}})$. The cohomology of a
differential graded algebra $H^*({\mathcal{A}})$ is naturally a DGA with the
product inherited from that on ${\mathcal{A}}$ and with the differential
being identically zero. The DGA $({\mathcal{A}},\,d_{\mathcal{A}})$ is {\it connected} if
$H^0({\mathcal{A}})\,=\,\mathbb{R}$, and ${\mathcal{A}}$ is {\em $1$-connected\/} if, in
addition, $H^1({\mathcal{A}})\,=\,0$. Henceforth we shall assume that all our DGAs are connected.
In our context, the main example of DGA is the de Rham complex $(\Omega^*(M),\,d)$
of a connected differentiable manifold $M$, where $d$ is the exterior derivative of $M$.

Morphisms between differential algebras are required to be degree
preserving algebra maps which commute with the differentials.
A morphism $f:(\mathcal{A}, \,d_{\mathcal{A}})\to (\mathcal{B},\,d_{\mathcal{B}})$ is a {\it quasi-isomorphism} if
the map induced in cohomology $f^*:H^*(\mathcal{A})\to H^*(\mathcal{B})$ is an isomorphism.

A differential algebra $(\mathcal{A}, d_{\mathcal{A}})$ is said to be {\it minimal\/} if:
\begin{enumerate}
 \item $\mathcal{A}$ is free as an algebra, that is, $\mathcal{A}$ is the free
 algebra $\bigwedge V$ over a graded vector space $V\,=\,\bigoplus V^i$, and
 \item there exists a collection of generators $\{ a_\tau,
 \tau\in I\}$, for some well ordered index set $I$, such that
 $|a_\mu|\,\leq\, |a_\tau|$ if $\mu\, < \,\tau$ and each
 $d_{\mathcal{A}} a_\tau$ is expressed in terms of preceding $a_\mu$ ($\mu\,<\,\tau$).
 This implies that $d_{\mathcal{A}}a_\tau$ does not have a linear part.
\end{enumerate}

We shall say that $(\mathcal {M},\,d_{\mathcal {M}})$ is a {\it minimal model} of the
differential algebra $(\mathcal{A}, d_{\mathcal{A}})$ if $(\mathcal {M},\,d_{\mathcal {M}})$ is a minimal DGA and there
exists a morphism of differential graded algebras
$$\rho\colon{(\mathcal {M},\,d_{\mathcal {M}})}\longrightarrow {(\mathcal{A}, d_{\mathcal{A}})}$$
 inducing an isomorphism
$\rho^*\colon H^*(\mathcal {M})\longrightarrow H^*(\mathcal {A})$ in cohomology.
In~\cite{Halperin}, Halperin proved that any connected differential algebra
$(\mathcal{A}, d_{\mathcal{A}})$ has a minimal model unique up to isomorphism. For $1$-connected
differential algebras, a similar result was proved by Deligne, Griffiths,
Morgan and Sullivan~\cite{DGMS,GM}.

A {\it minimal model\/} of a connected differentiable manifold $M$
is a minimal model $(\bigwedge V,\,d)$ for the de Rham complex
$(\Omega^*(M),\,d)$ of differential forms on $M$. If $M$ is a simply
connected manifold, then the dual of the real homotopy vector
space $\pi_i(M)\otimes \RR$ is isomorphic to $V^i$ for any $i$.
This relation also holds when $i\,>\,1$ and $M$ is nilpotent, that
is, the fundamental group $\pi_1(M)$ is nilpotent and its action
on $\pi_j(M)$ is nilpotent for all $j\,>\,1$ (see~\cite{DGMS}).

We say that a differential algebra $(\mathcal{A}, d_{\mathcal{A}})$
is a {\em model} of a differentiable manifold $M$ if $(\mathcal{A}, d_{\mathcal{A}})$
and $M$ have the same minimal model.

Recall that a minimal algebra $(\bigwedge V,\,d)$ is called
{\it formal} if there exists a
morphism of differential algebras $\psi\colon {(\bigwedge V,\,d)}\,\longrightarrow\,
(H^*(\bigwedge V),0)$ inducing the identity map on cohomology.
Also a differentiable manifold $M$ is called {\it formal\/} if its minimal model is
formal. Many examples of formal manifolds are known: spheres, projective
spaces, compact Lie groups, homogeneous spaces, flag manifolds,
and all compact K\"ahler manifolds.

The formality of a minimal algebra is characterized as follows.

\begin{proposition}[\cite{DGMS}]\label{prop:criterio1}
A minimal algebra $(\bigwedge V,\,d)$ is formal if and only if the space $V$
can be decomposed into a direct sum $V\,=\, C\oplus N$ with $d(C) \,=\, 0$,
and $d$ injective on $N$, such that every closed element in the ideal
$I(N)\, \subset\, \bigwedge V$ generated by $N$ is exact.
\end{proposition}

This characterization of formality can be weakened using the concept of
$s$-formality introduced in \cite{FM}.

\begin{definition}\label{def:primera}
A minimal algebra $(\bigwedge V,\,d)$ is $s$-formal
($s\,>\, 0$) if for each $i\,\leq\, s$
the space $V^i$ of generators of degree $i$ decomposes as a direct
sum $V^i\,=\,C^i\oplus N^i$, where the spaces $C^i$ and $N^i$ satisfy
the three following conditions:
\begin{enumerate}
\item $d(C^i) \,=\, 0$,

\item the differential map $d\,\colon\, N^i\,\longrightarrow\, \bigwedge V$ is
injective, and

\item any closed element in the ideal
$I_s\,=\,I(\bigoplus\limits_{i\leq s} N^i)$, generated by the space
$\bigoplus\limits_{i\leq s} N^i$ in the free algebra $\bigwedge
(\bigoplus\limits_{i\leq s} V^i)$, is exact in $\bigwedge V$.

\end{enumerate}
\end{definition}

A differentiable manifold $M$ is $s$-formal if its minimal model
is $s$-formal. Clearly, if $M$ is formal then $M$ is $s$-formal for all $s\,>\,0$.
The main result of \cite{FM} shows that sometimes the weaker
condition of $s$-formality implies formality.

\begin{theorem}[\cite{FM}]\label{fm2:criterio2}
Let $M$ be a connected and orientable compact differentiable
manifold of dimension $2n$ or $(2n-1)$. Then $M$ is formal if and
only if it is $(n-1)$-formal.
\end{theorem}

One can check that any simply connected compact manifold is $2$-formal. Therefore,
Theorem \ref{fm2:criterio2} implies that any simply connected compact manifold of
dimension not more than $6$ is formal. (This result was early proved by Neisendorfer and Miller in \cite{N-Miller}.)
For $7$-dimensional compact manifolds, we have that $M$ is formal if and only if $M$ is $3$-formal.
Moreover, if $M$ is simply connected we have:

\begin{lemma}\label{lem:$3$-formal}
Let $M$ be a $7$-dimensional simply connected compact manifold with $b_2(M)\,\leq 1$.
Then, M is $3$-formal and so formal.
\end{lemma}

\begin{proof}
Let $(\bigwedge V,\,d)$ be the minimal model of $M$. Write
$V^i\,=\,C^i \oplus N^i$, $i\,\leq\, 3$.  Suppose that $b_2(M)\,=\,1$.
Since $M$ is simply connected, we get
$C^1 \,=\, N^1\,=\,0$, $C^2\,=\,\langle a \rangle$, $N^2\,=\,0$ and $V^3\,=\,C^3 \oplus N^3$,
where $N^3$ has at most one element $x$ if $a^2$ defines the zero class
in the cohomology group $H^4(\bigwedge V,\,d)$. If $N^3\,=\,0$, then $M$ is clearly $3$-formal.
If  $N^3\,=\,\langle x \rangle$ with ${d}x\,=\,a^2$, then take $z\in I(N^{\leq 3})$ a closed
element in $\bigwedge V$. As $H^*(\bigwedge V)=H^*(M)$ has only cohomology in degrees $0,2,3,4,5,7$,
it must be $\deg z=5,7$. If $\deg z=5$ then $z=a\cdot\,x$ which is not closed, and if
$\deg z=7$ then $z=a^2\cdot\,x$ which is not closed either.
Thus,  according to Definition \ref{def:primera}, $M$ is $3$-formal, and by Theorem
\ref{fm2:criterio2}, $M$ is formal.

Finally, in the case that $b_2(M)\,=\,0$, then $C^i\,=N^i\,=\,0$, for $i\,=\,1, 2$, $V^3\,=\,C^3$ and $N^3\,=\,0$.
Hence, $M$ is formal.
\end{proof}

In order to detect non-formality, instead of computing the minimal
model, which usually is a lengthy process, one can use Massey
products, which are known to be obstructions to formality. The simplest type
of Massey product is the triple (also known as ordinary) Massey
product, which we define next.

Let $(\mathcal{A},\,d_{\mathcal{A}})$ be a DGA (in particular, it can be the de Rham complex
of differential forms on a differentiable manifold). Suppose that there are
cohomology classes $[a_i]\,\in\, H^{p_i}(\mathcal{A})$, $p_i\,>\,0$,
$1\,\leq\, i\,\leq\, 3$, such that $a_1\cdot a_2$ and $a_2\cdot a_3$ are
exact. Write $a_1\cdot a_2\,=\,d_{\mathcal{A}}x$ and $a_2\cdot a_3\,=\,d_{\mathcal{A}}y$.
The {\it (triple) Massey product} of the classes $[a_i]$ is defined to be
$$
\langle [a_1],[a_2],[a_3] \rangle \,=\,
[ a_1 \cdot y+(-1)^{p_{1}+1} x
\cdot a_3] \in \frac{H^{p_{1}+p_{2}+ p_{3} -1}(\mathcal{A})}{[a_1]\cdot
H^{p_{2}+ p_{3} -1}(\mathcal{A})+[a_3]\cdot H^{p_{1}+ p_{2} -1}(\mathcal{A})}\, .
$$

Note that a Massey product $\langle [a_1],[a_2],[a_3] \rangle$ on $(\mathcal{A},\,d_{\mathcal{A}})$
is zero (or trivial) if and only if there exist $\widetilde{x}, \widetilde{y}\in \mathcal{A}$ such that
$a_1\cdot a_2=d_{\mathcal{A}}\widetilde{x}$, \, $a_2\cdot a_3=d_{\mathcal{A}}\widetilde{y}$\,  and
$0=[ a_1 \cdot \widetilde{y}+(-1)^{p_{1}+1} \widetilde{x}\cdot a_3]\in H^{p_{1}+p_{2}+ p_{3} -1}(\mathcal{A})$.

We wil use also the following property.
\begin{lemma} \label{lemm:massey-models}
Let $M$ be a connected differentiable manifold. Then,  Massey products on $M$
can be calculated by using any model of $M$.
\end{lemma}

\begin{proof}
It is enough to prove that if $\varphi:(\mathcal{A},\,d_{\mathcal{A}}) \to (\mathcal{B},\,d_{\mathcal{B}})$
is a quasi-isomorphism, then $\varphi^*(\langle [a_1],[a_2],[a_3] \rangle)=\langle [a_1'],[a_2'],[a_3'] \rangle$, for
$[a_j']=\varphi^*([a_j])$. But this is clear, take
$a_1\cdot a_2=d_{\mathcal{A}}{x}$, \, $a_2\cdot a_3=d_{\mathcal{A}}{y}$\,  and let
 $$
  f=[ a_1 \cdot {y}+(-1)^{p_{1}+1} {x}\cdot a_3]\in
  \frac{H^{p_{1}+p_{2}+ p_{3} -1}(\mathcal{A})}{[a_1]\cdot
  H^{p_{2}+ p_{3} -1}(\mathcal{A})+[a_3]\cdot H^{p_{1}+ p_{2} -1}(\mathcal{A})}\,
 $$
be its Massey product $\langle [a_1],[a_2],[a_3] \rangle$. Then $a_j'=\varphi(a_j)$ satisfy
$a_1'\cdot a_2'=d_{\mathcal{B}}{x'}$, \, $a_2'\cdot a_3'=d_{\mathcal{B}}{y'}$, where
$x'=\varphi(x)$, $y'=\varphi(y)$. Therefore
 $$
  f'=[ a_1' \cdot {y'}+(-1)^{p_{1}+1} {x'}\cdot a_3'] =\varphi^*(f) \in
  \frac{H^{p_{1}+p_{2}+ p_{3} -1}(\mathcal{B})}{[a_1']\cdot
  H^{p_{2}+ p_{3} -1}(\mathcal{B})+[a_3']\cdot H^{p_{1}+ p_{2} -1}(\mathcal{B})}\,
 $$
is the Massey product $\langle [a_1'],[a_2'],[a_3'] \rangle$.
 \end{proof}

The existence of a non-zero Massey product is an obstruction to
the formality. We have the following result, initially proved in~\cite{DGMS}.

\begin{lemma} \label{lem:criterio1}
 If $M$ has a non-trivial Massey product then $M$ is non-formal.
\end{lemma}

\begin{proof}
 Suppose that $M$ is formal and let us see that all the Massey products are trivial.
 Let $a_1,a_2,a_3$ be cohomology classes on $M$ with $a_1\cdot a_2=a_2\cdot a_3=0$. By Lemma
 \ref{lemm:massey-models}, to compute the Massey we can use any model for $M$. By definition of formality,
$(H^*(M),0)$ is a model for $M$. In this model we can use $x=0, y=0$ for
$a_1\cdot a_2=d {x}$, $a_2\cdot a_3=d{y}$. So the Massey product is 
$\langle [a_1],[a_2],[a_3] \rangle=[ a_1 \cdot {y}+(-1)^{p_{1}+1} {x}\cdot a_3]=0$.
 \end{proof}

The concept of formality is also defined for CW-complexes which
have a minimal model $(\bigwedge V,\,d)$. Such a minimal model is
constructed as the minimal model associated to the differential
complex of piecewise-linear polynomial forms \cite{GM}. We shall
not need this in full generality, but we shall use the case when
$X$ is an orbifold. Thus, since the proof of Theorem \ref{fm2:criterio2} given in \cite{FM}
only uses that the cohomology $H^*(M)$ is a Poincar\'e duality algebra,
Theorem \ref{fm2:criterio2} also holds for compact connected orientable orbifolds.

%%%%%%%%%%%%%%%%%%%%%%%%%%%%%%%%%%%%%%%%%%%%%%%%%%%%%%%
%%%%%%%%%%%%%%%%%%%%%%%%%%%%%%%%%%%%%%%%%%%%%%%%%%%%%%%
%%%%%%%%%%%%%%%%%%%%%%%%%%%%%%%%%%%%%%%%%%%%%%%%%%%%%%%
\section{Formality of $3$-Sasakian manifolds}\label{$3$-sasaki:formal}
%%%%%%%%%%%%%%%%%%%%%%%%%%%%%%%%%%%%%%%%%%%%%%%%%%%%%%%
%%%%%%%%%%%%%%%%%%%%%%%%%%%%%%%%%%%%%%%%%%%%%%%%%%%%%%%

We recall the notion of $3$-Sasakian manifolds
following \cite{Blair,BG-2,BG}.
An odd dimensional Riemannian manifold $(N,g)$ is Sasakian if its 
cone $(N\times{\mathbb{R}}^+, g^c = t^2 g+dt^2)$ is K\"ahler, that
is the cone metric $g^c = t^2 g+dt^2$ admits a compatible
integrable almost complex structure $J$ so that
$(N\times{\mathbb{R}}^+, g^c = t^2 g+dt^2, J)$ is a K\"ahler
manifold. In this case the  Reeb vector field $\xi=J\partial_t$ is
a Killing vector field of unit length. The corresponding $1$-form
$\eta$ defined by $\eta(X)=g(\xi,X)$, for any vector field $X$ on
$N$, is a contact form. Let $\nabla$ be the Levi-Civita connection
of $g$. The (1,1) tensor $\phi X=\nabla_X\xi$ satisfies the
identities 
 $$
 \phi^2=-Id+\eta\otimes\xi, \quad g(\phi X,\phi Y)=g(X,Y)-\eta(X)\eta(Y), 
\quad d\eta(X,Y)=2g(\phi X,Y),
 $$
for vector fields $X,Y$.

A collection of three Sasakian structures on a
$(4n+3)$-dimensional Riemannian manifold satisfying
quaternionic-like identities form a $3$-Sasakian structure. More
precisely, a Riemannian manifold $(N, g)$ of dimension $4n+3$ is
called $3$-Sasakian if its cone $(N\times{\mathbb{R}}^+, g^c = t^2
g\,+\,dt^2)$ is hyperk\"ahler, that is the metric $g^c = t^2
g\,+\,dt^2$ admits three compatible integrable almost complex
structure $J_s$, $s=1,2,3$, satisfying the quaternionic relations,
i.e.,\ $J_1J_2=-J_2J_1=J_3$, such that $(N\times{\mathbb{R}}^+,
g^c = t^2 g\,+\,dt^2, J_1, J_2, J_3)$ is a hyperk\"ahler manifold.
Equivalently, the holonomy group of the cone metric $g^c$ is a
subgroup of $\mathrm{Sp}(n+1)$. In this case the Reeb vector
fields $\xi_s=J_s\partial_t$ $(s=1,2,3)$ are Killing vector
fields. The  three Reeb vector fields $\xi_s$, the three
$1$-forms  $\eta_s$ and the three $(1,1)$ tensors $\phi_s$, where
$s=1,2,3$, satisfy the relations
\begin{align*}
&\eta_i(\xi_j)=g(\xi_i,\xi_j)=\delta_{ij}, \\
&\phi_i\,\xi_j=-\phi_j\,\xi_i=\xi_k, \\
&\eta_i\circ\phi_j=-\eta_j\circ\phi_i=\eta_k,\\
&\phi_i\circ\phi_j-\eta_j\otimes\xi_i=-\phi_j\circ\phi_i+\eta_i\otimes\xi_j=\phi_k,
\end{align*}
for any cyclic permutation $(i, j, k)$ of $(1, 2, 3)$.

The  Reeb vector fields $\xi_s$ satisfy the relations
$[\xi_i,\xi_j]=2\xi_k$ thus spanning an integrable $3$-dimensional
distribution on a $3$-Sasakian manifold. In order to prove Theorem \ref{th:main}, we use the two following
results about the three dimensional $3$-Sasakian foliation
 proved by Boyer and Galicki in \cite{BG-2}.

\begin{proposition} [\cite{BG-2}]\label{prop: leaves}
Let $(N, g)$ be a $3$-Sasakian manifold such that the
{Reeb} vector fields $(\xi_1, \xi_2, \xi_3)$ are complete.
Denote by ${\mathcal{F}}$ the canonical three dimensional
foliation on $N$. Then,
\begin{enumerate}
\item[i)] The leaves of ${\mathcal{F}}$ are totally geodesic spherical space forms $\Gamma{\backslash}S^3$ of constant curvature
one, where $\Gamma\,\subset \mathrm{Sp}(1) =  \mathrm{SU}(2)$ is a finite subgroup.
\item[ii)] The $3$-Sasakian structure on $M$ restricts to a $3$-Sasakian structure on each leaf.
\item[iii)] The generic leaves are either $\mathrm{SU}(2)$ or $\mathrm{SO}(3)$.
\end{enumerate}
\end{proposition}

\begin{theorem} [\cite{BG-2}]\label{th:structure}
Let $(N, g)$ be a $3$-Sasakian manifold of dimension $4n + 3$
such that the {Reeb} vector fields $(\xi_1, \xi_2, \xi_3)$ are complete. 
Then the space of leaves $N /{\mathcal{F}}_3$ has the structure of
 a quaternionic K\"ahler orbifold
 $(\mathcal{O}, g_{\mathcal{O}})$ of dimension $4n$ such that the natural projection
 $\pi\colon N \,\longrightarrow\,{\mathcal{O}}$ is a principal V-bundle with group $\mathrm{SU}(2)$ or $\mathrm{SO}(3)$,
  and  $\pi$ is a Riemannian orbifold submersion such that the scalar curvature of
$g_{\mathcal{O}}$ is $16n(n + 2)$.
\end{theorem}

\bigskip

\noindent{\bf Proof of Theorem~ \ref{th:main}:}

Consider $(N, g)$ a $7$-dimensional simply connected compact $3$-Sasakian manifold whose second Betti number is
$b_{2}(N)\,\leq\,1$. Then  $N$ is formal by Lemma \ref{lem:$3$-formal}.
The converse is equivalent to prove that if the compact  $3$-Sasakian manifold
$(N, g)$ has $b_{2}(N)\,=\,k>1$, then $N$ is non-formal.
To this end, we will show that $N$ has a non-trivial Massey product.

Denote by $\mathcal{F}$ the canonical three dimensional foliation
on $N$. Since $N$ is compact, the {Reeb} vector fields
$(\xi_1, \xi_2, \xi_3)$ are complete. Then, by Proposition
\ref{prop: leaves}, the leaves of $\mathcal{F}$ are quotients
$\Gamma{\backslash}S^3$, where $\Gamma\,\subset \mathrm{Sp}(1) =
\mathrm{SU}(2)$ is a finite subgroup. Theorem \ref{th:structure}
implies that there is an orbifold $S^3$-bundle
 $S^3\,\longrightarrow\, N \,\longrightarrow\,{\mathcal{O}}$, where $\mathcal{O}$ is a compact quaternionic K\"ahler orbifold of
dimension $4$, with Euler class given by the integral cohomology class $\Omega\in H^4(\mathcal{O})$ of the quaternionic $4$-form.
Note that $\mathcal{O}$ is simply connected because $N$ is so
(see \cite[Theorem 4.3.18]{BG}). So  $S^3\,\longrightarrow\, N \,\longrightarrow\,{\mathcal{O}}$ is
a rational fibration (that is, after rationalization of the spaces, it becomes
a fibration). Therefore \cite{RS}, if $(\mathcal{A}, d_{\mathcal{A}})$ is a model of $\mathcal{O}$, then
$(\mathcal{A} \otimes \bigwedge(z), d)$, with $|z|\,=\,3$, $d\vert_{\mathcal{A}}\,=\,d_{\mathcal{A}}$
and $dz\,=\,\Omega$, is a model of $N$.

Moreover,  $\mathcal{O}$ is formal because it is a simply connected compact orbifold of dimension $4$ and
Theorem \ref{fm2:criterio2} also holds for orbifolds. Thus, a model of $\mathcal{O}$ is
$(H^*(\mathcal{O}), 0)$, where $H^*(\mathcal{O})$ is the
cohomology algebra of $\mathcal{O}$. Hence, a model of $N$ is the differential algebra
$(H^*(\mathcal{O})\otimes \bigwedge(z), d)$, with $dz=\Omega \in H^4(\mathcal{O})$, and
 $$
 H^1(\mathcal{O})\,=\,H^3(\mathcal{O}) =0\,, \qquad
 H^2(\mathcal{O})\,=\,\langle a_1,a_2,\cdots,a_k\rangle, \quad k\,\geq 2,
 $$
since $b_{2}(N)\,=\,k\,\geq 2$.
Since $H^*(\mathcal{O})$ is a Poincar\'e duality algebra, the intersection pairing is a non-degenerate
quadractic form on $H^2(\mathcal{O})$. Therefore, we can take $a_1,a_2,\ldots, a_k$ an orthogonal basis
of $H^2(\mathcal{O})$, that is $a_i\cdot a_j=0$ for $i\neq j$.
The cohomology of $N$ is
 $$
 \begin{array}{l}
 H^1(N)= H^3(N)=H^4(N)=H^6(N) =0\,, \\
 H^2(N)=\langle a_1,a_2,\cdots,a_k\rangle , \\
 H^5(N)=\langle a_1\, z,a_2\, z,\cdots,a_k\, z\rangle .
\end{array}
 $$
Then $a_1\cdot a_1 \, = \, \Omega=dz$.
Thus the Massey product $\langle a_1, a_1, a_i \rangle=  a_i\, z$ is defined for any $i\in\{1, 2, 3, 4\}$ and, for $i\,\not=\,1$, it is non-trivial.
\hfill$\square$

\subsection*{A $7$-dimensional formal Sasaki-Einstein manifold with $b_2\geq 2$}\label{sasa-einst:formal}
We show an example of  a $7$-dimensional simply connected, compact  Sasaki-Einstein manifold, with second Betti
number $b_2\geq 2$, which is formal. To this end, we recall the following

\begin{theorem} \cite{Munoz-Tralle} \label{thm:formal-$7$-dim}
 Let $N$ be a simply connected compact Sasakian $7$-dimensional manifold.
 Then $N$ is formal if and only if all triple Massey products are trivial.
\end{theorem}

Next we consider $M$ to be the blow up of the complex projective space
${\mathbb{CP}}{}^3$ at 4 points, that is,
 $$
 M\,=\,{\mathbb{CP}}{}^3\#\, \overline{\mathbb{CP}}{}^3\#\, \overline{\mathbb{CP}}{}^3\#\,
 \overline{\mathbb{CP}}{}^3\#\, \overline{\mathbb{CP}}{}^3,
$$
where $\overline{\mathbb{CP}}{}^3$ is ${\mathbb{CP}}{}^3$ with the opposite of the standard orientation.
The de Rham cohomology of $M$ is
\begin{itemize}
 \item $H^0(M)=\langle 1\rangle$,
 \item $H^1(M)=0$,
 \item $H^2(M)=\langle b, a_1, a_2, a_3, a_4 \rangle$,
 \item $H^3(M)=0$,
 \item $H^4(M)=\langle b^2, a_{1}^2, a_{2}^2, a_{3}^2, a_{4}^2\rangle$,
 \item $H^5(M)=0$,
 \item $H^6(M)=\langle b^3\rangle$,
\end{itemize}
where $b$ is the integral cohomology class defined by the  K\"ahler form $\omega$ on $\mathbb{CP}{}^3$.
Among these cohomology classes,
the following relations are satisfied
$$
b^3\,=\,-a_{1}^3 \,=\,-a_{2}^3\,=\,-a_{3}^3\,=\,-a_{4}^3, \qquad b\cdot a_{i} \,=\,0\,=\,a_{i}\cdot a_{j},
\qquad 1\,\leq i\,\leq 4, \quad i\,\not=\,j.
$$

 \begin{theorem} \label{th: formal-sasa-einst}
Let $N$ be the total space of the circle bundle  $S^1\,\longrightarrow\, N \,\longrightarrow\, M$, with Euler class
$\ell b\,-\,\sum_{i=1}^4 \, a_{i}$, where $\ell >0$ is a large integer.
Then $N$ is a simply connected, compact  Sasaki-Einstein manifold, with second Betti
number $b_2=4$, which is formal. Therefore, $N$ does not admit any $3$-Sasakian structure.
\end{theorem}

\begin{proof}
First, note that we can assume that
$\ell \, b\,-\,\sum_{i=1}^4 \, a_{i}$ is the integral cohomology class defined by the K\"ahler form on the complex
manifold
$M\,=\,{\mathbb{CP}}{}^3\#\, \overline{\mathbb{CP}}{}^3\#\, \overline{\mathbb{CP}}{}^3\#\,
 \overline{\mathbb{CP}}{}^3\#\, \overline{\mathbb{CP}}{}^3$, for $\ell$ large enough. Therefore
there is a circle bundle $N\longrightarrow M$ with Euler class equal to $\ell \, b\,-\,\sum_{i=1}^4 \, a_{i}$.

Clearly $N$ is a $7$-dimensional simply connected, compact
manifold, with second Betti number $b_2=4$. Moreover, $N$ is
Sasaki-Einstein. Indeed, the manifold $M$, that is the blow up of the
complex projective space at four points, is a toric symmetric Fano
manifold with vanishing Futaki invariant  \cite{Fu} and the existence of a
K\"ahler Einstein metric follows from \cite{BaS} (see also
\cite{WZ}). An application of \cite[Example 1]{Friedrich-Kath}
gives the Sasaki-Einstein structure on $N$.

Now, by Theorem
\ref{thm:formal-$7$-dim}, $N$ is formal if and only if all the
triple Massey products on $N$ are trivial.
By Lemma \ref{lemm:massey-models}, we know that Massey products on a manifold can be computed by using any model for
the manifold.
Since $M$ is a simply connected compact manifold of dimension $6$,  $M$ is formal. Thus, a model of $M$ is
$(H^*(M), 0)$, where $H^*(M)$ is the
cohomology algebra of $M$. Hence, a model of $N$ is the differential algebra
$(\mathcal{A}, d)$, where $\mathcal{A}\,=\,H^*(M)\otimes \bigwedge(x)$, $|x|\,=\,1$
and $dx=\ell\, b\,-\,a_{1}\,-\, a_{2}\,-\, a_{3}\,-\, a_{4}$. Then,
 $$
 \begin{array}{l}
 H^1(\mathcal{A}, d)\,=\,H^3(\mathcal{A}, d)=H^4(\mathcal{A}, d)\,=\,H^6(\mathcal{A}, d)=0, \\
 H^2(\mathcal{A}, d)\,=\,\langle a_1,a_2, a_3, a_4\rangle, \\
 H^5(\mathcal{A}, d)\,=\,\langle (\ell a_1^2 + b^2) x , (\ell a_2^2 + b^2) x , (\ell a_3^2 + b^2) x, (\ell a_4^2 + b^2) x \rangle.
 \end{array}
 $$

We note that if a Massey product of three cohomology classes of $H^*(\mathcal{A}, d)$ is defined, then at most one
of these cohomology classes has degree $\geq 3$ since $\dim N=7$. Thus, by dimension reasons,
the unique possible non-trivial Massey products are Massey products of the cohomology classes $a_i$ of degree $2$.
Clearly, for any $1\leq i\leq 4$, $\langle a_i, a_i, a_i \rangle\,=\,0$. Now
we consider $i$ and $j$ such that $1\leq i, j\leq 4$ and $i\,\not=\,j$.
Then, $a_i\cdot  a_i\,= \,a_{i}^2\,= \,- {d}(x\cdot  a_i)$ and $a_i\cdot  a_j\,= \,0$. Thus,
the triple Massey product $\langle a_i, a_i, a_j \rangle$ is defined and
 $$
 \langle a_i, a_i, a_j \rangle\,=\,- x\cdot  a_i\cdot  a_j\,=\,0, 
 $$
since $a_i\cdot  a_j\,= \,0$. Finally, if $i, j, k \in\{1, 2, 3, 4\}$ are such that $i\not=j\not=k\not=i$, then
$a_i\cdot  a_j\,= \,0\,= \,a_j\cdot a_k$. Hence, the Massey product $\langle a_i, a_j, a_k\rangle$ is trivial again,
which completes the proof.
\end{proof}

Finally, we show examples of $7$-dimensional simply connected compact Sasakian manifolds,
with second Betti number $b_{2}\geq 2$, which are formal. For this,
we consider $M$ to be the blow up of the complex projective space
${\mathbb{CP}}{}^3$ at $k$ points, with $k\geq 2$, that is
 $$
   M\,=\,{\mathbb{CP}}{}^3\#\, \overline{\mathbb{CP}}{}^3\#\,  \overbrace{\cdots}^{k}
  \, \#\,\overline{\mathbb{CP}}{}^3,
$$
where $\overline{\mathbb{CP}}{}^3$ is ${\mathbb{CP}}{}^3$ with the opposite of the standard orientation.
(Note that the case $k=4$ was considered in Theorem \ref{th: formal-sasa-einst}.)
Now, the de Rham cohomology of $M$ is
\begin{itemize}
 \item $H^0(M)=\langle 1\rangle$,
 \item $H^1(M)=0$,
 \item $H^2(M)=\langle b, a_1, a_2, \cdots, a_k \rangle$,
 \item $H^3(M)=0$,
 \item $H^4(M)=\langle b^2, a_{1}^2, a_{2}^2, \cdots a_{k}^2\rangle$,
 \item $H^5(M)=0$,
 \item $H^6(M)=\langle b^3\rangle$,
\end{itemize}
where $b$ is the integral cohomology class defined by the  K\"ahler form $\omega$ on $\mathbb{CP}{}^3$.
Among these cohomology classes,
the following relations are satisfied
$$
b^3\,=\,-a_{i}^3, \quad \text{for} \quad 1\leq i \leq k, \qquad b\cdot a_{i} \,=\,0\,=\,a_{i}\cdot a_{j},
\quad \text{for} \quad 1\,\leq i, j \,\leq k \quad \text{and} \quad i\,\not=\,j.
$$

 \begin{proposition}
Let $P$ be the total space of the circle bundle  $S^1\,\longrightarrow\, P \,\longrightarrow\, M$, with Euler class
$\ell b\,-\,\sum_{i=1}^k \, a_{i}$, where $\ell >0$ is a large integer.
Then $P$ is a simply connected, compact  Sasakian manifold, with second Betti
number $b_2=k$, which is formal. Therefore, for $k\geq 2$, $P$ does not admit any $3$-Sasakian structure.
\end{proposition}

\begin{proof}
Since,  for $\ell$ large enough, $\ell \, b\,-\,\sum_{i=1}^k \, a_{i}$ is the integral cohomology class defined by K\"ahler form on the complex
manifold
$M\,=\,{\mathbb{CP}}{}^3\#\, \overline{\mathbb{CP}}{}^3\#\, \cdots \,\#\,
\overline{\mathbb{CP}}{}^3$, we can consider
the principal circle bundle $S^1\,\longrightarrow\, P \longrightarrow M$ with Euler class equal to $\ell \, b\,-\,\sum_{i=1}^k \, a_{i}$.
Then $P$ is a $7$-dimensional simply connected, compact
Sasakian manifold, with second Betti number $b_2(P)=k$.

By Theorem
\ref{thm:formal-$7$-dim}, $P$ is formal if and only if all the
triple Massey products on $P$ are trivial.
Moreover, by Lemma \ref{lemm:massey-models}, to calculate Massey products on a manifold we can use any model for
the manifold.
Since $M$ is a simply connected compact manifold of dimension $6$,  $M$ is formal. Thus, a model of $M$ is
$(H^*(M), 0)$, where $H^*(M)$ is the
cohomology algebra of $M$. Hence, a model of $N$ is the differential algebra
$(\mathcal{A}, d)$, where $\mathcal{A}\,=\,H^*(M)\otimes \bigwedge(x)$, $|x|\,=\,1$
and $dx=\ell\, b\,-\,\sum_{i=1}^k \, a_{i}$. Then,
 $$
 \begin{array}{l}
 H^1(\mathcal{A}, d)\,=\,H^3(\mathcal{A}, d)=H^4(\mathcal{A}, d)\,=\,H^6(\mathcal{A}, d)=0, \\
 H^2(\mathcal{A}, d)\,=\,\langle a_1,a_2, \cdots, a_{k-1}, a_k\rangle, \\
 H^5(\mathcal{A}, d)\,=\,\langle (\ell a_i^2 + b^2) x , \, i= 1, 2, \cdots, k\rangle.
 \end{array}
 $$

Now a similar proof to that given in Theorem \ref{th: formal-sasa-einst} allows one to show that all triple Massey products on $P$ are zero.
\end{proof}

\begin{remark}
Note that the K\"ahler manifold $M$ defined as the blow up of the
complex projective space ${\mathbb{CP}}{}^3$ at $k$ points, with
$k\geq 2$, is K\"ahler Einstein if $k=4$. For $k \leq 3$, it
happens that the authomorfism group of $M$ is not reductive, and
the very well known Matsushima criterium \cite{Matsushima} implies
that $M$ does not admit K\"ahler Einstein metrics. If $k > 4$, it
is not clear (at least to the authors) whether the
manifold $M$, that is the blow up of ${\mathbb{CP}}{}^3$ at more
than 4 points, admits a K\"ahler Einstein metric. So we can only
claim that, for $k > 4$, the total space $P$ of the circle bundle
over $M$ is Sasakian and formal, hence $N$ does not admit any
$3$-Sasakian structure.
\end{remark}

\section*{Acknowledgements} We would like to thank V. Apostolov for explaining to us the criterium for the existence
of a K\"ahler Einstein metric on  Fano manifolds.  The first
author was partially supported through Project MICINN (Spain)
MTM2011-28326-C02-02 and MTM2014-54804-P. The
second author was partially supported by Contract DFNI
I02/4/12.12.2014 and Contract 148/17.04.2015 with the Sofia
University ``St.Kl.Ohridski''. The third author was partially
supported by Project MICINN (Spain) MTM2010-17389.

\end{document}